\documentclass[12pt, a4paper]{article}
\usepackage{amsfonts}
\usepackage{mathrsfs}
\usepackage{latexsym}
\usepackage{xy}
\usepackage{amsfonts,amsmath,amssymb,amsthm}
\usepackage{color}

\xyoption{all}

\newcommand{\bcen}{\begin{center}}     \newcommand{\ecen}{\end{center}}
\newcommand{\bay}{\begin{array}}      \newcommand{\eay}{\end{array}}
\newcommand{\beq}{\begin{eqnarray*}}      \newcommand{\eeq}{\end{eqnarray*}}

\def\dim{\mathrm{dim}}

\def\Ext{\mathrm{Ext}}

\def\gl{\mathrm{gl.dim}}

\def\HH{\mathrm{HH}}
\def\Hom{\mathrm{Hom}}

\def\id{\mathrm{id}}
\def\Im{\mathrm{Im}}

\def\inj{\mathrm{inj}}

\def\mod{\mathrm{mod}}
\def\Mod{\mathrm{Mod}}

\def\op{\mathrm{op}}
\def\pd{\mathrm{pd}}
\def\per{\mathrm{per}}
\def\proj{\mathrm{proj}}

\def\rad{\mathrm{rad}}

\def\RHom{\mathrm{RHom}}

\def\thick{\mathrm{thick}}

\begin{document}

\newtheorem{theorem}{Theorem}
\newtheorem{proposition}{Proposition}
\newtheorem{lemma}{Lemma}
\newtheorem{corollary}{Corollary}
\newtheorem{remark}{Remark}
\newtheorem{example}{Example}
\newtheorem{definition}{Definition}
\newtheorem*{conjecture}{Conjecture}
\newtheorem{question}{Question}

\title{\large\bf Eventually homological isomorphisms in recollements of derived categories}

\author{\large Yongyun Qin}

\date{\footnotesize College of Mathematics and Statistics,
Qujing Normal University, \\ Qujing, Yunnan 655011, China. E-mail:
qinyongyun2006@126.com }

\maketitle

\begin{abstract} For a recollement $(\mathcal{D}B,\mathcal{D}A,\mathcal{D}C)$ of derived categories
of algebras, we investigate when the functor $j^*:\mathcal{D}A\rightarrow\mathcal{D}C$
is an eventually homological isomorphism. In this context, we
compare the algebras $A$ and $C$ with respect to Gorensteinness, singularity categories
and the finite generation condition Fg for the Hochschild cohomology.
The results are applied to stratifying ideals,
triangular matrix algebras and
derived discrete algebras.
\end{abstract}

\medskip


\medskip


\section{\large Introduction}

\indent\indent Recollement of triangulated categories, introduced by
Beilinson et al. \cite{BBD82}, is an important tool
in algebraic geometry and representation theory. In particular,
a recollement $(\mathcal{D}B,\mathcal{D}A,\mathcal{D}C)$ of derived categories
of algebras provide a useful framework
for comparing the algebras $A$, $B$ and $C$ with respect to certain homological properties,
such as global dimension \cite{Wie91,Koe91,AKLY17},
finitistic dimension \cite{Hap93,CX17}, Hochschild homology and
cyclic homology \cite{Kel98}, Hochschild cohomology \cite{Han14},
Gorensteinness \cite{Pan13,QH16}, and so on.
Meanwhile, like recollement of triangulated categories,
recollement of abelian categories attracts a lot of attention in
recent years \cite{FP05,Psa14,PSS14,PJ14}. In particular,
recollement of abelian categories $(\mathcal{A},\mathcal{B},\mathcal{C})$
with the functor $e:\mathcal{B}\rightarrow \mathcal{C}$ being an
eventually homological isomorphism was used as a common context
to compare the Gorensteinness, singularity categories
and the Fg condition for the algebras
$A$ and $eAe$, where $e$ is an idempotent of $A$ \cite{PSS14}.
The motivation of this paper is to give a derived categories version of this work.

Let $(\mathcal{D}B,\mathcal{D}A,\mathcal{D}C)$ be a recollement of derived categories
of algebras. The functor $j^*:\mathcal{D}A \rightarrow \mathcal{D}C$
is called an {\it eventually homological isomorphism}
if there is an integer $t$ such that
for every pair of finitely generated right $A$-modules
$M$ and $N$, and every $j > t$, there is an isomorphism
$$\Hom_{\mathcal{D}A}(M,N[j])\cong \Hom_{\mathcal{D}C}(j^*M,j^*N[j])$$
of abelian groups. Our first main theorem characterizes
when the functor $j^*:\mathcal{D}A \rightarrow \mathcal{D}C$ in
a recollement $(\mathcal{D}B,\mathcal{D}A,\mathcal{D}C)$
is an eventually homological isomorphism.

\medskip

 {\bf Theorem A.} {\it Let $A$, $B$ and $C$ be finite dimensional algebras over an algebraically
closed field $k$, and let ($\mathcal{D} B$,\ $\mathcal{D} A$,\ $\mathcal{D} C$,\ $i^*,i_*,i^!,j_!,j^*,j_*$)
be a standard recollement defined
by $X \in D^b(C^{op} \otimes A)$ and $Y \in D^b(A^{op} \otimes B)$.
Suppose $X^*=\RHom_A(X,A)$ and $Y^*=\RHom_B(Y,B)$.
Then the following are equivalent:

{\rm (a)} The functor $j^*$ is an eventually homological isomorphism;

{\rm (b)} $\gl B<\infty$, $_AY\in K^b(\proj A^{op})$ and $Y_A^*\in K^b(\proj A)$;

{\rm (b')} $\gl B<\infty$, $_CX\in K^b(\proj C^{op})$ and $X_C^*\in K^b(\proj C)$;

{\rm (c)} $\RHom_B(Y,Y)\in K^b(\proj A^e)$, where $A^e=A^{op}\otimes_kA$.
}

\medskip

Here, we refer \cite{Han14}
or Definition 1 for the concept of standard recollement.
In order to describe our second theorem, we recall the following
three definitions briefly.
A finite dimensional algebra $A$ is said
to be {\it Gorenstein} if $\id_AA < \infty $ and $\id_{A^{op}}A < \infty$;
The {\it singularity category} of $A$ is
defined to be the Verdier quotient $D^b(\mod A)/ K^b(\proj A)$,
and two algebras are said to be {\it singularly equivalent} if there is
a triangle equivalent between their singularity categories;
$A$ is said to {\it satisfy the Fg condition} if the
Hochschild cohomology ring $HH^*(A)$ is Noetherian
and the Yoneda algebra $\Ext _A^*(A/ \rad A, A/ \rad A)$ is a finitely generated
$HH^*(A)$-module
(for more details and backgrounds, see Section 4.3).
Our second theorem shows that recollement $(\mathcal{D}B,\mathcal{D}A,\mathcal{D}C)$
with the functor $j^*:\mathcal{D}A\rightarrow\mathcal{D}C$ being an eventually homological isomorphism
is a very good context
to compare the algebras $A$ and $C$.

\medskip

 {\bf Theorem B.} {\it Let $A$, $B$ and $C$ be finite dimensional algebras over an algebraically
closed field $k$, and let ($\mathcal{D} B$,\ $\mathcal{D} A$,\ $\mathcal{D} C$,
\ $i^*,i_*,i^!,j_!,j^*,j_*$)
be a recollement such that the functor $j^*$ is an
eventually homological isomorphism.
Then the following assertions hold true:

{\rm (a)} $A$ is Gorenstein if and only if so is $C$;

{\rm (b)} The algebras $A$ and $C$ are singularly equivalent;

{\rm (c)} $A$ satisfies Fg if and only if so does $C$.
}

\medskip

Applying Theorem B to recollements induced by idempotents,
we recover a result of Nagase, where the algebras $A$ and $eAe$
are compared, for an idempotent $e$ and a stratifying ideal $AeA$
\cite{Nag11}. Also, we recover some relevant results in
triangular matrix algebras. Finally, we show that
derived discrete algebras can be reduced to $k$ or
$2$-truncated cycle algebras, via recollements of derived categories
with the functor $j^*$ being an
eventually homological isomorphism. As an application,
we prove that derived discrete algebras satisfy the Fg condition.

The paper is organized as follows: In section 2, we will recall
some notions and results on recollements of derived categories.
Section 3 is about eventually homological isomorphisms in recollements
of derived categories, in which Theorem A is obtained.
In section 4, we
will prove Theorem B. In section 5,
we apply our main theorem to stratifying ideals,
triangular matrix algebras and
derived discrete algebras.

\section{\large Definitions and conventions}
\begin{definition}{\rm (Beilinson-Bernstein-Deligne \cite{BBD82})
Let $\mathcal{T}_1$, $\mathcal{T}$ and $\mathcal{T}_2$ be
triangulated categories. A {\it recollement} of $\mathcal{T}$
relative to $\mathcal{T}_1$ and $\mathcal{T}_2$ is given by
$$\xymatrix@!=4pc{ \mathcal{T}_1 \ar[r]^{i_*=i_!} & \mathcal{T} \ar@<-3ex>[l]_{i^*}
\ar@<+3ex>[l]_{i^!} \ar[r]^{j^!=j^*} & \mathcal{T}_2
\ar@<-3ex>[l]_{j_!} \ar@<+3ex>[l]_{j_*}}$$
and denoted by ($\mathcal{T}_1$, $\mathcal{T}$, $\mathcal{T}_2$,\ $i^*,i_*,i^!,j_!,j^*,j_*$)
(or just ($\mathcal{T}_1$, $\mathcal{T}$, $\mathcal{T}_2$))
such that

(R1) $(i^*,i_*), (i_*,i^!), (j_!,j^*)$ and $(j^*,j_*)$ are adjoint
pairs of triangle functors;

(R2) $i_*$, $j_!$ and $j_*$ are full embeddings;

(R3) $j^*i_*=0$ (and thus also $i^!j_*=0$ and $i^*j_!=0$);

(R4) for each $X \in \mathcal {T}$, there are triangles

$$\begin{array}{l} j_!j^*X \rightarrow X  \rightarrow i_*i^*X  \rightarrow
\\ i_!i^!X \rightarrow X  \rightarrow j_*j^*X  \rightarrow
\end{array}$$ where the arrows to and from $X$ are the counits and the
units of the adjoint pairs respectively. }
\end{definition}

Let $k$ be a field, $D:=\Hom _k(-, k)$ and $ \otimes
:= \otimes _k$. Throughout the paper, all algebras are assumed to be
finite dimensional
algebras over $k$. Let $A$ be such an algebra.
Denote by $\Mod A$ the
category of right $A$-modules, and by $\mod A$ (resp. $\proj A$
and $\inj A$) its full subcategories consisting of all finitely
generated modules (resp. finitely generated projective
modules and injective modules).  For
$* \in \{{\rm nothing}, b \}$, denote by $\mathcal{D}^*(\Mod
A)$ (resp. $\mathcal{D}^*(\mod A)$) the derived category of
(cochain) complexes of objects in $\Mod A$ (resp. $\mod A$)
satisfying the corresponding boundedness condition. Denote by
$K^b(\proj A)$ the homotopy category of
bounded complexes of objects in $\proj A$.
Up to isomorphism, the objects in $K^{b}(\proj A)$ are
precisely all the compact objects in $\mathcal{D}(\Mod A)$. For
convenience, we do not distinguish $K^{b}(\proj A)$ from the {\it
perfect derived category} $\mathcal{D}_{\per}(A)$ of $A$, i.e., the
full triangulated subcategory of $\mathcal{D} A$ consisting of all
compact objects, which will not cause any confusion. Moreover, we
also do not distinguish $K^b(\inj A)$ and
$\mathcal{D}^b(\mod A)$ from their essential images under the
canonical full embeddings into $\mathcal{D}(\Mod A)$. Usually, we
just write $\mathcal{D} A$ (resp. $D^b(A)$) instead of $ \mathcal{D}(\Mod A)$
(resp. $\mathcal{D}^b(\mod A)$). In this paper,
all functors between triangulated categories are assumed
to be triangulated functors.

\begin{definition}\label{def-standard-recollement}{\rm (Han \cite{Han14})
Let $A,B$ and $C$ be algebras.
An recollement ($\mathcal{D} B$,\ $\mathcal{D} A$,\ $\mathcal{D} C$,\ $i^*,i_*,i^!,j_!,j^*,j_*$) is
said to be {\it standard} and {\it defined by} $Y \in \mathcal
{D}^b(A^{\op} \otimes B)$ and $X \in \mathcal {D}^b(C^{\op} \otimes A)$
if $i^* \cong -\otimes^L_A Y$ and $j_! \cong -\otimes^L_CX$.
}
\end{definition}

\begin{proposition}\label{prpo-functor} {\rm (Han \cite{Han14})
Let $A,B$ and $C$ be algebras, and ($\mathcal{D} B$,\ $\mathcal{D} A$,\ $\mathcal{D} C$,\ $i^*,i_*,i^!,j_!,j^*,j_*$)
a standard recollement defined by $Y \in \mathcal
{D}^b(A^{\op} \otimes B)$ and $X \in \mathcal {D}^b(C^{\op} \otimes A)$. Then
$$\begin{array}{ll}
i^*\cong -\otimes^L_A Y, & j_! \cong -\otimes^L_CX, \\
i_*=\RHom _B(Y, -)=-\otimes^L_B Y^*, & j^*=\RHom _A(X, -)=-\otimes^L_A X^*, \\
i^!=\RHom _A(Y^*, -), & j_*=\RHom _C(X^*, -),\\
\end{array}$$
where $X^*=\RHom_A(X,A)$ and $Y^*=\RHom_B(Y,B)$.
}
\end{proposition}

Two recollements ($\mathcal{T}_1$, $\mathcal{T}$, $\mathcal{T}_2$, $i^*,i_*,i^!,j_!,j^*,j_*$)
and ($\mathcal{T}_1'$, $\mathcal{T}'$, $\mathcal{T}_2'$, ${i^*}',i_*',{i^!}',j_!',
\linebreak {j^*}',j_*'$)
are said to be {\it equivalent} if $(\Im j_!, \Im i_*, \Im j_*)=(\Im j_!', \Im i_*', \Im j_*')$.
From \cite[Proposition 3 and Remark 1]{Han14}, every recollement of derived categories
of algebras is equivalent to a standard one.

\section{\large Proof of Theorem A}

\indent\indent Let $A$ and $B$ be two algebras.
Given a functor $F:\mathcal{D}A \rightarrow \mathcal{D}B$, $F$
is called an {\it eventually homological isomorphism}
if there is an integer $t$ such that
for every pair of objects $M$ and $N$ in $\mod A$, and every $j > t$, there is an isomorphism
$$\Hom_{\mathcal{D}A}(M,N[j])\cong \Hom_{\mathcal{D}B}(FM,FN[j])$$
of abelian groups. This definition is taken from \cite[Section 3]{PSS14}
with a minor modification.

In this section, we will characterizes
when the functor $j^*:\mathcal{D}A \rightarrow \mathcal{D}C$ in
a recollement $(\mathcal{D}B,\mathcal{D}A,\mathcal{D}C)$
is an eventually homological isomorphism, that is, we will
prove Theorem A.
This result is used in Section 4 for comparing Gorensteinness,
singularity categories and the Fg condition of the algebras
$A$ and $C$. Let's start with the following lemmas.

\begin{lemma}\label{lemma-ehi}
Assume that $F:\mathcal{D}A \rightarrow \mathcal{D}B$ is an eventually homological isomorphism.
Let $X$, $Y \in D^b(\mod A)$ such that $H^i(X)=H^i(Y)=0$,
for any $i<m$ or $i>n$. Then there is an integer $t$ such that
$\Hom_{\mathcal{D}A}(X, Y[j])
\cong \Hom_{\mathcal{D}B}(FX, FY[j])$, for
every $j>t$.
\end{lemma}

\begin{proof}
Since $F$ is an eventually homological isomorphism, there
exists some $t_0$ such that
$$\Hom_{\mathcal{D}A}(M,N[j])\cong \Hom_{\mathcal{D}B}(FM,FN[j]),$$
for any $M, N \in \mod A$ and every $j > t_0$.
Up to quasi-isomorphism, we assume that $X$ and $Y$
are of the form
$$X:\ \ 0 \longrightarrow X^{m} \longrightarrow
X^{m+1} \longrightarrow \cdots \longrightarrow X^{n} \longrightarrow 0 ,$$
$$Y:\ \ 0 \longrightarrow Y^{m} \longrightarrow
Y^{m+1} \longrightarrow \cdots \longrightarrow Y^{n} \longrightarrow 0 .$$
Using truncation technique just like \cite[Lemma 1.6]{Kato98}, we can
prove that $\Hom_{\mathcal{D}A}(X, Y[j])
\cong \Hom_{\mathcal{D}B}(FX, FY[j])$, for
every $j>t_0+n-m$.
\end{proof}

\begin{lemma}\label{lemma-restrict-bound}
Let $A$, $B$ and $C$ be finite dimensional
algebras over a field $k$, and let
($\mathcal{D} B$,\ $\mathcal{D} A$,\ $\mathcal{D} C$,\ $i^*,i_*,i^!,j_!,j^*,j_*$)
be a recollement.
Then the following hold true.

{\rm (1)} For every $M\in \mod B$, there exist two integers $m$ and $n$ such that
$H^i(i_*M)=0$, for any $i<m$ or $i>n$.

{\rm (2)} If $i_*$ restricts to $K^b(\proj)$, that is,
$i_*$ sends $K^b(\proj B)$ to $K^b(\proj A)$, then there exist two integers $m_1$ and $n_1$ such that
$H^i(i^!M)=0$, for any $i<m_1$ or $i>n_1$, and every $M\in \mod A$.

{\rm (3)} If $i_*$ restricts to $K^b(\inj)$, then there exist two integers $m_2$ and $n_2$ such that
$H^i(i^*M)=0$, for any $i<m_2$ or $i>n_2$, and every $M\in \mod A$.

\end{lemma}

\begin{proof}
(1): By \cite[Lemma 2.9 (e)]{AKLY17}, the functor
$i^*$ restricts to $K^b(\proj)$, and thus, $i^*A\in K^b(\proj B)$.
Assume $i^*A$ is quasi-isomorphic to a projective complex $P^\bullet$ of the form:
$$0 \longrightarrow P^{-n} \longrightarrow
P^{-n+1} \longrightarrow \cdots \longrightarrow P^{-m} \longrightarrow 0 .$$
For any $M\in \mod B$ and $i\in \mathbb{Z}$, we have
$$H^i(i_*M)\cong \Hom _{\mathcal{D} A}(A, i_*M[i])
\cong \Hom _{\mathcal{D} B}(i^*A, M[i])\cong \Hom _{\mathcal{D} B}(P^\bullet, M[i]).$$
Therefore, $H^i(i_*M)=0$, for any $i<m$ or $i>n$.

(2): This can be proved in a similar way as we did in (1).

(3): Assume $i_*(DA)$ is quasi-isomorphic to a injective complex $I^\bullet$ of the form:
$$0 \longrightarrow I^{-n_2} \longrightarrow
I^{-n_2+1} \longrightarrow \cdots \longrightarrow I^{-m_2} \longrightarrow 0 .$$
For any $M\in \mod A$ and $i\in \mathbb{Z}$, we have
$$DH ^i(i^*M)\cong
H ^{-i} (D(i^*M)) \cong
\Hom _{\mathcal{D}k}(i^*M, k[-i]) \cong
\Hom _{\mathcal{D}A}(i^*M, DA[-i]).$$
Here, the last isomorphism follows by adjunction. Using
adjointness again, we get
$$DH ^i(i^*M) \cong \Hom _{\mathcal{D}A}(i^*M, DA[-i]) \cong
\Hom _{\mathcal{D}B}(M, i_*DA[-i])$$
Therefore, $H^i(i^*M)=0$, for any $i<m_2$ or $i>n_2$.
\end{proof}

Now we are ready to prove Theorem A, which is divided into
Theorem 1 and Theorem 2.

\begin{theorem}\label{theorem-main-1}
Let $A$, $B$ and $C$ be finite dimensional
algebras over a field $k$, and let
($\mathcal{D} B$,\ $\mathcal{D} A$,\ $\mathcal{D} C$,\ $i^*,i_*,i^!,j_!,j^*,j_*$)
be a recollement.
Then the following are equivalent:

{\rm (a)} The functor $j^*$ is an eventually homological isomorphism;

{\rm (b)} $\gl B<\infty$, and $i_*$ restricts to both $K^b(\proj)$ and $K^b(\inj)$;

{\rm (b')} $\gl B<\infty$, and $j^*$ restricts to both $K^b(\proj)$ and $K^b(\inj)$.
\end{theorem}
\begin{proof}
(a)$\Rightarrow$ (b):
For any $M$, $M'\in \mod B$, and any $i\in \mathbb{N}$, we have
$$\Ext_B^i(M,M')\cong \Hom _{\mathcal{D} B}(M, M'[i])
\cong\Hom _{\mathcal{D} A}(i_*M, i_*M'[i]).$$
By \cite[Lemma 2.9 (e)]{AKLY17}, we have $i_*M, i_*M'
\in D^b(\mod A)$. Therefore, Lemma~\ref{lemma-ehi} and Lemma~\ref{lemma-restrict-bound} (1)
yield that there exists some integer $t$ such that
$$ \Hom _{\mathcal{D} A}(i_*M, i_*M'[i])\cong \Hom _{\mathcal{D} C}(j^*i_*M, j^*i_*M'[i]),
\ \ \forall \ i>t .$$
Since $j^*i_*=0$, we obtain $\Ext_B^i(M,M')\cong0$, for any $i>t $.
Therefore, $\gl B< \infty$.

Now we claim $i_*$ restricts to $K^b(\proj)$, and the statement
$i_*$ restricts to $K^b(\inj)$ can be proved dually.
For any $P\in K^b(\proj B)$, we want to show
$i_*P\in K^b(\proj A)$. For this, since $i_*P\in D^b(\mod A)$, it is equivalent to
show that for any simple $A$-module $S$, there are only
finite many integers $n$ such that $\Hom _{\mathcal{D} A}
(i_*P,S[n])\neq 0$ (see the proof of \cite[Lemma 2.4 (c)]{AKLY17}).
Clearly,
$\Hom _{\mathcal{D} A}
(i_*P,S[n])= 0$ for sufficiently small $n$.
On the other hand, by Lemma~\ref{lemma-ehi}, there exists some integer $t$ such that
$\Hom _{\mathcal{D} A}
(i_*P,S[n])\cong \Hom _{\mathcal{D} C}
(j^*i_*P,j^*S[n])$, for any $n>t$. Since $j^*i_*=0$, we obtain that $\Hom _{\mathcal{D} A}
(i_*P,S[n])\cong0$, for any $n>t$.
Therefore, $i_*P\in K^b(\proj A)$.

\medskip

(b)$\Rightarrow$ (a): For any $M, N\in \mod B$ and $i\in\mathbb{N}$, applying the functor
$\Hom _{\mathcal{D} A}(-,N[i])$ to the triangle
$ j_!j^*M\rightarrow M\rightarrow i_*i^*M\rightarrow $,
we get exact sequence
$$\Hom _{\mathcal{D} A}
(i_*i^*M,N[i])\rightarrow \Hom _{\mathcal{D} A}
(M,N[i])\rightarrow \Hom _{\mathcal{D} A}
(j_!j^*M ,N[i]).$$
By Lemma~\ref{lemma-restrict-bound}, there exist
$m_1,m_2, n_1,n_2\in\mathbb{Z}$ such that $H^i(i^!M)=0$,
for any $i<m_1$ or $i>n_1$, and
$H^i(i^*M)=0$, for any $i<m_2$ or $i>n_2$. Since $\gl B<\infty$,
it follows from \cite[Lemma 1.6]{Kato98} that
there is an integer $t$ such that $\Hom _{\mathcal{D} B}
(i^*M,i^!N[i])=0$, for any $i>t$. Using adjointness,
we have $\Hom _{\mathcal{D} A}
(i_*i^*M,N[i])\cong 0$, for any $i>t$. Therefore,
$\Hom _{\mathcal{D} A}
(M,N[i])\cong\Hom _{\mathcal{D} A}
(j_!j^*M ,N[i])\cong \Hom _{\mathcal{D} A}
(j^*M ,j^*N[i])$, for any $i>t$. That is,
$j^*$ is an eventually homological isomorphism.

\medskip

(b)$\Leftrightarrow$ (b'): It follows from
\cite[Lemma 2.5 and Lemma 4.3]{AKLY17} that $i_*$ restricts to $K^b(\proj)$
if and only if $j^*$ restricts to $K^b(\proj)$.
Dually, $i_*$ restricts to $K^b(\inj)$
if and only if $j^*$ restricts to $K^b(\inj)$.
\end{proof}

\begin{corollary}\label{cor-ehi-stand}
Let $A$, $B$ and $C$ be finite dimensional
algebras over a field $k$, and let
($\mathcal{D} B$,\ $\mathcal{D} A$,\ $\mathcal{D} C$,\ $i^*,i_*,i^!,j_!,j^*,j_*$)
be a recollement such that the functor $j^*$ is an
eventually homological isomorphism. Then there is a standard recollement
($\mathcal{D} B$,\ $\mathcal{D} A$,\ $\mathcal{D} C$, ${i^*}',i_*',{i^!}',j_!',
{j^*}',j_*'$) such that the functor ${j^*}'$ is an
eventually homological isomorphism.
\end{corollary}
\begin{proof}
Since $j^*$ is an
eventually homological isomorphism, it follows from
Theorem~\ref{theorem-main-1} that $\gl B<\infty$, and $i_*$ restricts to both $K^b(\proj)$ and $K^b(\inj)$.
Therefore, the recollement
($\mathcal{D} B$,\ $\mathcal{D} A$,\ $\mathcal{D} C$,\ $i^*,i_*,i^!,j_!,j^*,j_*$)
can be extended one-step upwards and one-step downwards,
see \cite[Lemma 3 and Lemma 4]{QH16}.
On the other hand, it follows from \cite[Proposition 3 and Remark 1]{Han14}
that ($\mathcal{D} B$,\ $\mathcal{D} A$,\ $\mathcal{D} C$,\ $i^*,i_*,i^!,j_!,j^*,j_*$)
is equivalent to a standard recollement ($\mathcal{D} B$,\ $\mathcal{D} A$,
$\mathcal{D} C$, ${i^*}',i_*',{i^!}',j_!',
{j^*}',j_*'$), which can also be extended one-step upwards and one-step downwards.
Therefore, $i_*'$ restricts to both $K^b(\proj)$ and $K^b(\inj)$.
Using Theorem~\ref{theorem-main-1} again, we obtain that the functor ${j^*}'$ is an
eventually homological isomorphism.
\end{proof}

Owing to Corollary~\ref{cor-ehi-stand}, we will restrict
our discussions on standard recollement in the following text.

\begin{theorem}\label{theorem-main-2}
Let $A$, $B$ and $C$ be finite dimensional
algebras over a field $k$, and let
($\mathcal{D} B$,\ $\mathcal{D} A$,\ $\mathcal{D} C$,\ $i^*,i_*,i^!,j_!,j^*,j_*$)
be a standard recollement defined
by $X \in D^b(C^{op} \otimes A)$ and $Y \in D^b(A^{op} \otimes B)$.
Suppose $X^*=\RHom_A(X,A)$ and $Y^*=\RHom_B(Y,B)$.
Then the following are equivalent:

{\rm (a)} The functor $j^*$ is an eventually homological isomorphism;

{\rm (b)} $\gl B<\infty$, $_AY\in K^b(\proj A^{op})$ and $Y_A^*\in K^b(\proj A)$;

{\rm (b')} $\gl B<\infty$, $_CX\in K^b(\proj C^{op})$ and $X_C^*\in K^b(\proj C)$.

Moreover, if $k$ is algebraically
closed, these occur precisely when

{\rm (c)} $\RHom_B(Y,Y)\in K^b(\proj A^e)$, where $A^e=A^{op}\otimes_kA$.
\end{theorem}

\begin{proof}
(a)$\Leftrightarrow$ (b): It follows from \cite[Lemma 2.8]{AKLY17} that
$_AY\in K^b(\proj A^{op})$ if and only if $i^*\cong -\otimes^L_A Y$
restricts to $D^b(\mod)$, and this occurs precisely when
$i_*$ restricts to $K^b(\inj )$, see
\cite[Lemma 4]{QH16}. By \cite[Lemma 2.5]{AKLY17}
and Proposition~\ref{prpo-functor},
$Y_A^*\in K^b(\proj A)$ if and only if $i_*\cong -\otimes^L_B Y^*$ restricts to $K^b(\proj )$.
Now the statement follows from Theorem~\ref{theorem-main-1}.

\medskip

(a)$\Leftrightarrow$ (b'):
As above, we obtain that $_CX\in K^b(\proj C^{op})$
if and only if
$j^*$ restricts to $K^b(\inj )$, and $X_C^*\in K^b(\proj C)$
if and only if $j^*$ restricts to $K^b(\proj )$.
Thus, the statement follows from Theorem~\ref{theorem-main-1}.

\medskip

(b)$\Leftrightarrow$ (c):
According to \cite[Theorem 1]{Han14} and \cite[Theorem 2]{Han14}, a
recollement of derived categories of algebras induces those of
tensor product algebras and opposite algebras respectively. Therefore, we
have the following two recollements:
$$\xymatrix @R=0.6in @C=0.8in{
\mathcal{D}(B^e)
\ar[r]|{F_1} & \mathcal{D}(A^{\op} \otimes _kB) \ar@<+3ex>[l]
\ar@<-3ex>[l]|{L_1} \ar[r] & \mathcal{D}(C^{\op} \otimes _kB)
\ar@<+3ex>[l] \ar@<-3ex>[l] }$$

$$\xymatrix @R=0.6in @C=0.8in{
\mathcal{D}(A^{\op} \otimes_k B)
\ar[r]|{F_2} & \mathcal{D}(A^e) \ar@<+3ex>[l]
\ar@<-3ex>[l]|{L_2} \ar[r] & \mathcal{D}(A^{\op} \otimes _kC),
\ar@<+3ex>[l] \ar@<-3ex>[l] }$$
where $L_1
\cong Y^{*} \otimes_A^L -$, $F_1 \cong Y \otimes_B^L -$, $L_2
\cong - \otimes_A^L Y$, $F_2 \cong - \otimes_B^L Y^{*}$.

Now we claim (b)$\Rightarrow$ (c). Since $B$ is a finite dimensional
algebra over an algebraically
closed field, the condition $\gl B<\infty$ is equivalent to
$B\in K^b(\proj B^{e})$ (Ref. \cite[Lemma 7.2]{Rou08}).
On the other hand, $_AY\in K^b(\proj A^{op})$ and $Y_A^*\in K^b(\proj A)$
implies that both $F_1$ and $F_2$ preserve compactness.
Therefore, $F_2F_1(B)\cong Y\otimes _B^LY^*\cong\RHom_B(Y,Y)\in K^b(\proj A^e)$.

Next, we prove (c)$\Rightarrow$ (b).
Because $F_2F_1(B)\cong\RHom_B(Y,Y)\in K^b(\proj A^e)$ and both
$L_1$ and $L_2$ preserve compactness (Ref. \cite[Lemma 2.9 (e)]{AKLY17}),
we have that $B\cong L_1F_1B\cong L_1L_2F_2F_1B\in K^b(\proj B^e)$.
By \cite[Lemma 7.2]{Rou08}, we get $\gl B<\infty$.
Due to \cite[Lemma 4.2]{AKLY17}, $F_2F_1(B)\in K^b(\proj A^e)$ yields
that $F_1(B)\in K^b(\proj (A^{\op} \otimes_k B))$, that is,
$_AY_B\in K^b(\proj (A^{\op} \otimes_k B))$. Therefore, we get $_AY\in K^b(\proj A^{op})$.
Since $i^*A=Y_B$ is a compact generator of $\mathcal{D}B$, it follows
that $\thick Y_B=\thick B$, where $\thick Y_B$ is the smallest
triangulated subcategory of $\mathcal{D}B$ containing $Y_B$
and closed under direct summands.
By d\'{e}vissage, $\RHom_B(_AY_B,Y_B)\in K^b(\proj A)$ yields that
$\RHom_B(_AY_B,B)\in K^b(\proj A)$, that is, $Y_A^*\in K^b(\proj A)$.

\end{proof}

\section{\large Proof of Theorem B}
\indent\indent In this section, we will compare the Gorensteinness,
singularity categories and the Fg condition of the algebras
$A$ and $C$, where there is a recollement ($\mathcal{D} B$,\ $\mathcal{D} A$,\ $\mathcal{D} C$)
such that the functor $j^*:\mathcal{D}A\rightarrow\mathcal{D}C$ is an
eventually homological isomorphism.
\subsection{\large Comparison on Gorensteinness}

\indent\indent Recall that a finite dimensional algebra $A$ is said
to be {\it Gorenstein} if $\id_AA < \infty $ and $\id_{A^{op}}A < \infty$.

\begin{definition}{\rm (\cite{QH16})
Let $\mathcal{T}_1$, $\mathcal{T}$ and $\mathcal{T}_2$ be
triangulated categories, and $n$ a positive integer. An {\it
$n$-recollement} of $\mathcal{T}$ relative to $\mathcal{T}_1$ and
$\mathcal{T}_2$ is given by $n+2$ layers of triangle functors
$$\xymatrix@!=4pc{ \mathcal{T}_1 \ar@<+1ex>[r] \ar@<-3ex>[r]_\vdots & \mathcal{T}
\ar@<+1ex>[r]\ar@<-3ex>[r]_\vdots \ar@<-3ex>[l] \ar@<+1ex>[l] &
\mathcal{T}_2 \ar@<-3ex>[l] \ar@<+1ex>[l]}$$ such that every
consecutive three layers form a recollement.}
\end{definition}

We mention that the ideal of this definition comes from \cite{BGS88, AKLY17},
where the concept ``ladder'' was introduced to study mixed categories.
In terms of $n$-recollement, the relationship between recollement
of derived categories and the Gorensteinness of algebras are expressed
as follows.
\begin{proposition}{\rm (See \cite[Theorem III]{QH16})}\label{prop-nreco-Gor}
Let $A$, $B$ and $C$ be finite dimensional algebras, and $\mathcal{D} A$ admit an
$n$-recollement relative to $\mathcal{D} B$ and $\mathcal{D} C$.

{\rm (1)} $n = 3$: if $A$ is Gorenstein then so are $B$ and $C$;

{\rm (2)} $n \geq 4$: $A$ is Gorenstein if and only if so are $B$ and $C$.

\end{proposition}

\begin{lemma}\label{lemma-ehi-5rec}
Let $A$, $B$ and $C$ be finite dimensional
algebras over a field $k$, and let
($\mathcal{D} B$,\ $\mathcal{D} A$,\ $\mathcal{D} C$,\ $i^*,i_*,i^!,j_!,j^*,j_*$)
be a recollement such that the functor $j^*$ is an eventually homological isomorphism.
Then this recollement can be extended to a $5$-recollement
of $\mathcal{D} A$ relative to $\mathcal{D} B$ and $\mathcal{D} C$.
\end{lemma}

\begin{proof}
Since $j^*$ is an
eventually homological isomorphism, it follows from
Theorem~\ref{theorem-main-1} that $\gl B<\infty$, and $i_*$ restricts to both $K^b(\proj)$ and $K^b(\inj)$.
Therefore, the recollement
($\mathcal{D} B$,\ $\mathcal{D} A$,\ $\mathcal{D} C$,\ $i^*,i_*,i^!,j_!,j^*,j_*$)
can be extended one-step upwards and one-step downwards,
see \cite[Lemma 3 and Lemma 4]{QH16}.
Thus, we obtain a $3$-recollement
$$\xymatrix@!=4pc{ \mathcal{D}B \ar[r]|{i_* } \ar@<+4ex>[r]
\ar@<-4ex>[r]
& \mathcal{D}A \ar@<-2ex>[l]|{i^* } \ar@<+4ex>[r]
\ar@<-4ex>[r]
\ar@<+2ex>[l]|{i^!} \ar[r]|{j^* } &
\mathcal{D}C \ar@<-2ex>[l]|{j_! }
\ar@<+2ex>[l]|{j_*} } \eqno {\rm (R)}.$$
Since $\gl B<\infty$, it follows from \cite[Lemma 2.9 (e)]{AKLY17} that $i^*(DA)
\in D^b(\mod B)= K^b(\inj B)$, and by \cite[Lemma 4]{QH16},
(R) can be extend one step upwards.
Similarly, we have $i^!A\in D^b(\mod B)= K^b(\proj B)$,
and thus, (R) can be extend one step downwards.
\end{proof}

Now we get the main result of this section.
\begin{theorem}\label{theorem-Gor}
Let $A$, $B$ and $C$ be finite dimensional algebras over
a field $k$, and let ($\mathcal{D} B$,\ $\mathcal{D} A$,\ $\mathcal{D} C$,
\ $i^*,i_*,i^!,j_!,j^*,j_*$)
be a recollement such that the functor $j^*$ is an
eventually homological isomorphism.
Then the algebra $A$ is Gorenstein if and only if so is $C$.
\end{theorem}
\begin{proof}
According to Theorem~\ref{theorem-main-1}, we have $\gl B<\infty$,
and thus, $B$ is a Gorenstein algebra. Then it follows from
Lemma~\ref{lemma-ehi-5rec} and Proposition~\ref{prop-nreco-Gor}
that $A$ is Gorenstein if and only if so is $C$.
\end{proof}

\subsection{\large Comparison on singular categories}
\indent\indent Let $A$ be a finite dimensional algebra
over $k$. The {\it singularity category} \cite{Orl04} of $A$ is
defined to be the following Verdier quotient category: $$D_{sg}(A):= D^b(\mod A)/ K^b(\proj A).$$
Clearly, the singularity category $D_{sg}(A)$ carries a triangulated
structure, and $D_{sg}(A)=0$ if and only if $\gl A<\infty $.

From \cite{Chen14,Chen16}, two algebras are said to be {\it singularly equivalent} if there is
a triangle equivalent between $D_{sg}(A)$ and $D_{sg}(B)$.

\begin{proposition}{\rm (See \cite[Proposition 2.5]{LL15})}\label{prop-quot-rec}
Let ($\mathcal{D}_1$, $\mathcal{D}$, $\mathcal{D}_2$,\ $i^*,i_*,i^!,j_!,j^*,j_*$) be
a recollement of triangulated categories and $\mathcal{T}$ be a thick subcategory of $\mathcal{D}$.
Set $\mathcal{T}_1=i^*\mathcal{T}$ and $\mathcal{T}_2=j^*\mathcal{T}$.
If $i_*\mathcal{T}_1 \subseteq \mathcal{T}$ and $j_*\mathcal{T}_2 \subseteq \mathcal{T}$,
then there exists an induced recollement of triangulated quotient categories
($\mathcal{D}_1/\mathcal{T}_1 $, $\mathcal{D}/\mathcal{T}$, $\mathcal{D}_2/\mathcal{T}_1$).
\end{proposition}

\begin{proposition}\label{prop-sing-rec}
Let $A$, $B$ and $C$ be finite dimensional algebras, and $\mathcal{D} A$ admit a
$4$-recollement relative to $\mathcal{D} B$ and $\mathcal{D} C$.
Then there exists an induced recollement of singularity categories
($D_{sg}(C)$, $D_{sg}(A)$, $D_{sg}(B)$).
\end{proposition}
\begin{proof}
Let $$\xymatrix@!=6pc{ \mathcal{D}B \ar@<2.4ex>[r]|{j_1} \ar@<-4ex>[r]
\ar@<-0.8ex>[r]|{j_3} & \mathcal{D}A \ar@<-4ex>[l]
\ar@<-0.8ex>[l]|{j_2} \ar@<+2.4ex>[l] \ar@<-0.8ex>[r]|{i_3}
\ar@<2.4ex>[r]|{i_1} \ar@<-4ex>[r] & \mathcal{D}C \ar@<-4ex>[l]
\ar@<-0.8ex>[l]|{i_2} \ar@<2.4ex>[l]}$$ be a 4-recollement
of $\mathcal{D} A$ relative to $\mathcal{D} B$ and $\mathcal{D} C$.
By \cite[Lemma 2.9 (e)]{AKLY17}, this 4-recollement restricts to
the following recollement of $D^b(\mod)-$ level
$$\xymatrix @R=0.6in @C=0.8in{ D^b(\mod C) \ar[r]|{i_2} &  D^b(\mod A) \ar@<-1.6ex>[l]|{i_1}
\ar@<+1.6ex>[l]|{i_3} \ar[r]|{j_2} &  D^b(\mod B)
\ar@<-1.6ex>[l]|{j_1} \ar@<+1.6ex>[l]|{j_3}},$$
where all the sixes functors restrict to $K^b(\proj)$.
Therefore, we have that $i_1(K^b(\proj A))=K^b(\proj C)$, $j_2(K^b(\proj A))=K^b(\proj B)$,
$i_2(K^b(\proj C))\subseteq K^b(\proj A)$ and $j_3(K^b(\proj B))\subseteq K^b(\proj A)$.
Now the statement follows from Proposition~\ref{prop-quot-rec}.
\end{proof}

Now we get the main result of this section.
\begin{theorem}\label{theorem-sing}
Let $A$, $B$ and $C$ be finite dimensional algebras over
a field $k$, and let ($\mathcal{D} B$,\ $\mathcal{D} A$,\ $\mathcal{D} C$,
\ $i^*,i_*,i^!,j_!,j^*,j_*$)
be a recollement such that the functor $j^*$ is an
eventually homological isomorphism.
Then $j^*$ induces a singularly equivalent between $A$ and $C$.
\end{theorem}
\begin{proof}
According to the proof of Lemma~\ref{lemma-ehi-5rec}, there is a $4$-recollement
$$\xymatrix@!=6pc{ \mathcal{D}C \ar@<2.4ex>[r]|{j_!} \ar@<-4ex>[r]
\ar@<-0.8ex>[r]|{j_*} & \mathcal{D}A \ar@<-4ex>[l]
\ar@<-0.8ex>[l]|{j^*} \ar@<+2.4ex>[l] \ar@<-0.8ex>[r]|{i^!}
\ar@<2.4ex>[r]|{i^*} \ar@<-4ex>[r] & \mathcal{D}B \ar@<-4ex>[l]
\ar@<-0.8ex>[l]|{i_*} \ar@<2.4ex>[l]}.$$
From Proposition~\ref{prop-sing-rec}, there exists an induced recollement of singularity categories
($D_{sg}(B)$, $D_{sg}(A)$, $D_{sg}(C)$). On the other hand,
it follows from Theorem~\ref{theorem-main-1} that $\gl B<\infty$,
that is, $D_{sg}(B)=0$. Therefore, the functor $j^*$ induces a singularly equivalent between $A$ and $C$.
\end{proof}

\subsection{\large Comparison on Fg condition}
\indent\indent Let $A$ be a $k$-algebra and $X$ a complex
of $A$-module, then we define
$$\mathcal{E} _A^*(X) = \oplus _{n\in \mathbb{Z}}\Hom _{\mathcal{D}A} (X,X[n]).$$
Clearly, $\mathcal{E} _A^*(X)$ is a graded $k$-algebra with multiplication given by Yoneda product.
For some $d\in \mathbb{Z}$, we consider the graded ideals of the form
$$\mathcal{E} _A^{\geq d}(X) = \oplus _{n\geq d}\Hom _{\mathcal{D}A} (X,X[n]).$$
From \cite{CE56}, the Hochschild cohomology ring of $A$ is the extension ring $\HH ^*(A)
:= \mathcal{E} _{A^e}^*(A)$, where $A^e := A^{op} \otimes_k A$ is the enveloping algebra.
For convenience, we denote $\HH ^{\geq d}(A):= \mathcal{E} _{A^e}^{\geq d}(A) =
\oplus _{n\geq d}\Hom _{\mathcal{D}(A^e)} (A,A[n])$.

To describe the finite generation condition Fg, we first need to define a
$\HH ^*(A)$-module structure on $\mathcal{E} _A^*(X)$, for any complex $X$ in $D^b(A)$.
Indeed, this module structure is given by
the graded ring homomorphism $\varphi _X : \HH ^*(A) \rightarrow \mathcal{E} _A^*(X)$,
where $\varphi _X= X\otimes _A^L-$.

Support varieties for modules over artin algebras were defined by
Snashall and Solberg in \cite{SS04}, using the Hochschild cohomology ring.
In \cite{EHSST04}, Erdmann et al. introduced some finiteness
conditions (Fg1) and (Fg2) for an algebra $A$, which ensure
many results for support varieties over a group algebra also
hold for support varieties over a selfinjective algebra.
Later, these conditions were called
Fg and were studied by many authors \cite{KPS16, Nag11, PSS14, Ska16}.

\begin{definition}
Let $A$ be an algebra over a field $k$. We say that $A$ satisfies the Fg condition if the following is
true:

{\rm (Fg1)} The ring $\HH ^*(A)$ is noetherian.

{\rm (Fg2)} The $\HH ^*(A)$-module $\mathcal{E} _A^*(A/\rad A)$ is finitely generated.
\end{definition}

Nowadays, the Fg condition is becoming an important property
in geometry and representation theory
--- it is a good criterion for deciding whether a given algebra
has a nice theory of support varieties. What's more, the Fg condition
turns out to be related with Gorensteinness --- an algebra $A$
is Gorenstein if $A$ satisfies the Fg condition \cite{EHSST04}.
Therefore, it is of great interest to know whether the Fg
condition holds for various algebras, and to find out which
relations between algebras preserve the Fg condition.
The second question was considered in \cite{Nag11, PSS14} for algebras $A$ and $eAe$
with $e$ being an idempotent, in \cite{Lin11} for separable equivalence between
symmetric algebras, in \cite{Ska16} for singular equivalence between Gorenstein algebras
and in \cite{KPS16} for general derived equivalence.
In this section, we will consider algebras whose
derived categories are related by a recollement of triangulated categories.
The following propositions will be used.

\begin{proposition}{\rm (See \cite[Proposition 10.3 ]{Sol06})}\label{prop-def-fg}
If an artin algebra $A$ satisfies the Fg condition, then
$\mathcal{E} _A^*(X)$
is a finitely generated $\HH ^*(A)$-module, for every $X\in D^b(A)$.

\end{proposition}

\begin{proposition}\label{prop-comp-fg}
Let $A$ and $B$ be finite-dimensional $k$-algebras. Set $ M = A/ \rad A$
and $ N = B/ \rad B$.
Assume that we have the following two commutative diagrams
$$\xymatrix@!=1pc{ \HH ^{\geq d}(A) \ar[d]_f ^\cong
\ar[rr]^{\varphi _M} &&\mathcal{E} _A^{\geq d}(M)\ar[d]_g ^\cong \\
\HH ^{\geq d}(B)
\ar[rr]^{\varphi _Y} &&\mathcal{E} _B^{\geq d}(Y)
} \begin{array}{c}
\\  \\ \mbox{and} \\ \end{array}
\xymatrix@!=1pc{ \HH ^{\geq d}(A) \ar[d]_{f'} ^\cong
\ar[rr]^{\varphi _X} &&\mathcal{E} _A^{\geq d}(X)\ar[d]_{g'} ^\cong \\
\HH ^{\geq d}(B)
\ar[rr]^{\varphi _N} &&\mathcal{E} _B^{\geq d}(N)
}
$$
of graded nonunital $k$-algebras, for some positive integer $d$, some $X\in D^b(A)$ and $Y\in D^b(B)$,
where the vertical maps $f$, $g$, $f'$ and $g'$ are isomorphisms.
Then $A$ satisfies Fg if and only if so does $B$.
\end{proposition}
\begin{proof}
This follows from Proposition~\ref{prop-def-fg} and \cite[Proposition 6.3]{PSS14}.

\end{proof}
The following lemma is essentially due to \cite[Lemma 2.1]{KN09} and
\cite[Lemma 5]{Han14}.
\begin{lemma}\label{lemma-comm}
Let $A$, $B$ be an algebras, and let $\xymatrix@!=0.3pc{X \ar[r]^{u} & Y
\ar[r]^{v} & Z\ar[r] &}$ be a triangle in
$\mathcal{D} A$ such that $Y\in D^b(A)$, $Z\in K^b(\proj A)$ and $Z\in K^b(\inj A)$.
Assume that $F: \mathcal{D}A \rightarrow \mathcal{D}B$ is a triangulated functor
 such that $FY\in D^b(B)$, $FZ\in K^b(\proj B)$ and $FZ\in K^b(\inj B)$. Then
there is the following commutative diagram
 $$\xymatrix@!=3pc{ \mathcal{E} _A^{\geq d}(Y) \ar[d]_F
\ar[rr]^{\psi}_\cong &&\mathcal{E} _A^{\geq d}(X)\ar[d]_F \\
\mathcal{E} _B^{\geq d}(FY)
\ar[rr]^{{\psi}'} _\cong &&\mathcal{E} _B^{\geq d}(FX)
}$$ of graded nonunital $k$-algebras, for some positive integer $d$,
where the horizontal maps $\psi$ and $\psi '$ are isomorphisms.

\end{lemma}
\begin{proof}
For the sake of simplicity, we just denote the bifunctor $\Hom _{\mathcal{D}A}(-,-)$ and
$\Hom _{\mathcal{D}B}(-,-)$ by $(-,-)$, when it may not cause any confusion.
Applying the functor $(-, Y[n])$
(resp. $(-, FY[n])$ ) to
the triangle $\xymatrix@!=0.3pc{X \ar[r]^{u} & Y
\ar[r]^{v} & Z\ar[r] &}$ (resp. $\xymatrix{FX \ar[r]^{Fu} & FY
\ar[r]^{Fv} & FZ\ar[r] &}$), we have the following commutative diagram
$$\xymatrix{(Z, Y[n]) \ar[r]^-{v^*} \ar[d]_{F} & (Y, Y[n])
\ar[r]^{u^*} \ar[d]_{F} & (X, Y[n]) \ar[r] \ar[d]_{F} & (Z[-1], Y[n]) \ar[d]_{F} \\
(FZ, FY[n]) \ar[r]^-{(Fv)^*} & (FY, FY[n])
\ar[r]^-{(Fu)^*} & (FX, FY[n]) \ar[r] & (FZ[-1], FY[n]).
} $$
Since $Z\in K^b(\proj A)$, $Y\in D^b(A)$, $FZ\in K^b(\proj B)$
and $FY\in D^b(B)$,
there exists some integer $s$ such that
$(Z, Y[n])\cong 0 \cong (Z[-1], Y[n])$ and
$(FZ, FY[n])\cong 0 \cong (FZ[-1], FY[n])$, for any $n\geq s$.
Therefore, we have the following commutative diagram
$$\xymatrix{(Y, Y[n])
\ar[r]^{u^*}_\cong \ar[d]_{F} & (X, Y[n]) \ar[d]_{F} \\
(FY, FY[n])
\ar[r]^-{(Fu)^*}_\cong & (FX, FY[n]),
}   \ \ \ \ \ \ \ \ \begin{array}{c}
\\  \\ \mbox{(I)} \\ \end{array}$$
for any integer $n\geq s$,
where the horizontal maps $u^*$ and $(Fu)^*$ are isomorphisms.

Similarly, applying the functor $(X, -[n])$
(resp. $(FX, -[n])$) to
the triangle $\xymatrix@!=0.3pc{X \ar[r]^{u} & Y
\ar[r]^{v} & Z\ar[r] &}$ (resp. $\xymatrix{FX \ar[r]^{Fu} & FY
\ar[r]^{Fv} & FZ\ar[r] &}$), we have the following commutative diagram
$$\xymatrix{(X, Z[n-1]) \ar[r] \ar[d]_{F} & (X, X[n])
\ar[r]^{(u[n])_*} \ar[d]_{F} & (X, Y[n]) \ar[r] \ar[d]_{F} & (X, Z[n]) \ar[d]_{F} \\
(FX, FZ[n-1]) \ar[r]  & (FX, FX[n])
\ar[r]^{(Fu[n])_*}  & (FX, FY[n]) \ar[r]  & (FX, FZ[n]).
} $$
On the other hand, $Z\in K^b(\proj A)$ and $Y\in D^b(A)$
implies that $X\in D^b(A)$, and $FZ\in K^b(\proj B)$
and $FY\in D^b(B)$ implies that $FX\in D^b(B)$.
Since $Z\in K^b(\inj A)$ and $FZ\in K^b(\inj B)$,
there exists some integer $t$ such that
$(X, Z[n-1])\cong 0 \cong (X, Z[n])$ and
$(FX, FZ[n-1])\cong 0 \cong (FX, FZ[n])$, for any $n\geq t$.
Therefore, we have the following commutative diagram
$$\xymatrix{(X, X[n])
\ar[r]^{(u[n])_*}_\cong \ar[d]_{F} & (X, Y[n]) \ar[d]_{F}\\
(FX, FX[n])
\ar[r]^{(Fu[n])_*}_\cong  & (FX, FY[n]),
} \ \ \ \ \ \ \ \ \begin{array}{c}
\\  \\ \mbox{(II)} \\ \end{array}$$
for any integer $n\geq t$,
where the horizontal maps $(u[n])_*$ and $(Fu[n])_*$ are isomorphisms.
Let $d= \max \{1, t,s\}$. Combining (I) and (II), we have commutative diagram
$$\xymatrix{(Y,Y[n])
\ar[r]^\psi_\cong \ar[d]_{F} & (X, X[n]) \ar[d]_{F}\\
(FY, FY[n])
\ar[r]^{{\psi}'} _\cong  & (FX, FX[n]),
} $$
for any integer $n\geq d\geq 1$,
where the horizontal maps $\psi$ and ${\psi}'$ are isomorphisms.
Now the statement holds obviously.

\end{proof}

\begin{lemma}\label{lemma-func-M}
Let $M$ be a complex of $D^b(A)$. Then the functor
$\varphi _M =M\otimes _A^L-$
sends $K^b(\proj A^e)$ to $K^b(\proj A)$.
\end{lemma}
\begin{proof}
Assume $M$ is of the form $ 0 \longrightarrow M^{k} \longrightarrow
M^{k+1} \longrightarrow \cdots \longrightarrow M^{l} \longrightarrow 0 $
with $M^i\in \mod A$, and
$P: 0 \longrightarrow P^{p} \longrightarrow
P^{p+1} \longrightarrow \cdots \longrightarrow P^{q} \longrightarrow 0 $
with $P^i\in \proj A^e$. Then $\varphi _M(P) =M\otimes _A^LP\cong
(0 \longrightarrow X^{k+p} \longrightarrow
X^{k+p+1} \longrightarrow \cdots \longrightarrow X^{l+q} \longrightarrow 0)$,
where $X^n=\oplus _{i+j=n}M^i\otimes _AP^j$.
Clearly, $M^i\otimes _A(A\otimes _kA)\cong M^i\otimes _kA\in \proj A$,
and thus, $M^i\otimes _AP^j\in \proj A$. Therefore,
$\varphi _M(P)\in K^b(\proj A)$.
\end{proof}

\begin{lemma}\label{lemma-ff}
Let ($\mathcal{D} B$,\ $\mathcal{D} A$,\ $\mathcal{D} C$,\ $i^*,i_*,i^!,j_!,j^*,j_*$)
be a standard recollement defined
by $X \in D^b(C^{op} \otimes A)$ and $Y \in D^b(A^{op} \otimes B)$.
Then the functor $X^*\otimes _C^L-\otimes_C^LX :\mathcal{D} (C^e)
\longrightarrow \mathcal{D} (A^e)$ is fully faithful.
\end{lemma}
\begin{proof}
By \cite[Theorem 1 and Theorem 2]{Han14}, there are two recollements
$$\xymatrix @R=0.6in @C=0.8in{
\mathcal{D}(C^{\op} \otimes _k B)
\ar[r]& \mathcal{D}(C^{\op} \otimes _kA) \ar@<+3ex>[l]
\ar@<-3ex>[l] \ar[r] & \mathcal{D}(C^{e})
\ar@<+3ex>[l] \ar@<-3ex>[l]|{G_1} \ \ \ \mbox{and}}$$

$$\xymatrix @R=0.6in @C=0.8in{
\mathcal{D}(B^{\op} \otimes_k A)
\ar[r] & \ \ \ \ \mathcal{D}(A^e) \ar@<+3ex>[l]
\ar@<-3ex>[l] \ar[r] & \ \ \ \ \mathcal{D}(C^{\op} \otimes _kA)
\ar@<+3ex>[l] \ar@<-3ex>[l]|{G_2} },$$
where $G_1
\cong -\otimes_C^L X$ and $G_2 \cong X^* \otimes_C^L -$.
Since $G_1$ and $G_2$ are fully faithful, we have that
the functor $G_2G_1\cong X^*\otimes _C^L-\otimes_C^LX :\mathcal{D} (C^e)
\longrightarrow \mathcal{D} (A^e)$ is fully faithful.
\end{proof}

Now we are ready to compare the Fg condition in the
framework of recollement.
\begin{theorem}\label{theorem-Fg}
Let $A$, $B$ and $C$ be finite dimensional
algebras over an algebraically closed field $k$, and let
($\mathcal{D} B$,\ $\mathcal{D} A$,\ $\mathcal{D} C$,
\ $i^*,i_*,i^!,j_!,j^*,j_*$)
be a recollement such that the functor $j^*$ is an
eventually homological isomorphism.
Then we have

{\rm (1)} $\HH ^{\geq d}(A)\cong \HH ^{\geq d}(C)$,
for some positive integer $d$.

{\rm (2)} $A$ satisfies Fg if and only if so does $C$.
\end{theorem}

\begin{proof}
(1): Due to Corollary~\ref{cor-ehi-stand}, we may assume that
the recollement ($\mathcal{D} B$,\ $\mathcal{D} A$,\ $\mathcal{D} C$,\ $i^*,i_*,i^!,j_!,j^*,j_*$)
is standard and defined
by $X \in D^b(C^{op} \otimes A)$ and $Y \in D^b(A^{op} \otimes B)$.
Let $X^*=\RHom_A(X,A)$ and $Y^*=\RHom_B(Y,B)$.
From \cite[Corollary 3]{Han14}, there is a long exact sequence
$$\cdots \rightarrow \Hom _{D(A^e)}(\RHom _B(Y,Y),A[n])\rightarrow
\HH ^n(A)\rightarrow \HH ^n(C)\rightarrow \cdots .$$
On the other hand, it follows from Theorem~\ref{theorem-main-2}
that $\RHom _B(Y,Y)\in K^b(\proj A^e)$, and thus, there exists
some integer $d$ such that $\Hom _{D(A^e)}(\RHom _B(Y,Y),A[n])
\linebreak \cong 0$,
for any $n\geq d$. As a result, we get $\HH ^{n}(A)\cong \HH ^{n}(C)$,
for any $n\geq d$.

(2): Assume either $A$ or $C$ satisfies Fg. Then, it follows from
\cite[Theorem 1.5(a)]{EHSST04} that either $A$ or $C$
is Gorenstein, and by Theorem~\ref{theorem-Gor},
both $A$ and $C$
are Gorenstein.
From \cite[Theorem 1]{Han14},
we have the following recollement
$$\xymatrix @R=0.6in @C=0.8in{
\mathcal{D}(A^{\op} \otimes_k B)
\ar[r]|{I_*} & \mathcal{D}(A^e) \ar@<+3ex>[l]
\ar@<-3ex>[l]|{I^*} \ar[r]|{J^*} & \mathcal{D}(A^{\op} \otimes _kC)
\ar@<+3ex>[l] \ar@<-3ex>[l]|{J_!} }$$
where $I^*
\cong - \otimes_A^L Y$, $I_* \cong - \otimes_B^L Y^{*}$,
$J_!
\cong - \otimes_C^L X$ and $J^* \cong - \otimes_A^L X^{*}$.
Therefore, there is a triangle
$\begin{array}{l} X^* \otimes _C^LX \rightarrow A  \rightarrow Y\otimes_B^LY^*  \rightarrow
\end{array}$ in $\mathcal{D}(A^e)$.
By Theorem~\ref{theorem-main-2}, $Y\otimes_B^LY^*\cong \RHom _B(Y,Y)
\in K^b(\proj A^e)$, and by \cite[Lemma 2.1]{BJ13}, the Gorensteinness of $A$
implies the Gorensteinness of $A^e$. Hence, we have
$Y\otimes_B^LY^*\in K^b(\inj A^e)$.
For any $M\in D^b(A)$, consider the functor $\varphi _M =M\otimes _A^L-
:\mathcal{D}(A^e)\longrightarrow \mathcal{D}(A)$.
Clearly, $\varphi _M(A)=M\otimes _A^LA\cong M \in D^b(A)$,
and by Lemma~\ref{lemma-func-M}, $\varphi _M(Y\otimes_B^LY^*)
\in K^b(\proj A)=K^b(\inj A)$.
Applying Lemma~\ref{lemma-comm} to the triangle
$\begin{array}{l} X^* \otimes _C^LX \rightarrow A  \rightarrow Y\otimes_B^LY^*  \rightarrow
\end{array}$ and the functor $\varphi _M$, we obtain
the following commutative diagram
 $$\xymatrix{ \HH ^{\geq d}(A) \ar[d]_{\varphi _M}
\ar[rr]^{\psi}_\cong &&\mathcal{E} _{A^e}^{\geq d}(X^* \otimes _C^LX)\ar[d]^{\varphi _M} \\
\mathcal{E} _A^{\geq d}(M)
\ar[rr]^{{\psi}'} _\cong &&\mathcal{E} _A^{\geq d}(M\otimes_A^LX^* \otimes _C^LX)
}$$ of graded nonunital $k$-algebras, for some positive integer $d$,
where the horizontal maps $\psi$ and $\psi '$ are isomorphisms.
On the other hand, the associativity of the tensor product yields the
commutative diagram
$$\xymatrix{\mathcal{E} _{A^e}^{\geq d}(X^* \otimes _C^LX)
\ar[d]_{\varphi _M} && \HH ^{\geq d}(C)
\ar[ll]_{X^*\otimes _C^L-\otimes_C^LX}^\cong \ar[d]^{\varphi _{M\otimes_A^LX^*}} \\
\mathcal{E} _A^{\geq d}(M\otimes_A^LX^* \otimes _C^LX) &&
\mathcal{E} _C^{\geq d}(M\otimes_A^LX^* ),\ar[ll]_{-\otimes_C^LX}^\cong
}$$ where the horizontal two maps are isomorphisms,
because both the two functors are fully faithful, see Lemma~\ref{lemma-ff}.
As a result, we obtain the following commutative diagram
$$\xymatrix{ \HH ^{\geq d}(A) \ar[d]_{\varphi _M}
\ar[rr]^\cong &&\HH ^{\geq d}(C)\ar[d]^{\varphi _{M\otimes_A^LX^*}} \\
\mathcal{E} _A^{\geq d}(M)
\ar[rr]^\cong &&\mathcal{E} _C^{\geq d}(M\otimes_A^LX^* )
}$$ of graded nonunital $k$-algebras,
where the horizontal maps are isomorphisms.
Now take $M=A/\rad A$ and $M=C/\rad C\otimes_C^LX$ respectively,
we obtain two desired commutative diagrams in
Proposition~\ref{prop-comp-fg}. Therefore,
$A$ satisfies Fg if and only if so does $B$.
\end{proof}

\section{\large Applications}

\indent\indent In this section, we will
apply our main theorem to stratifying ideals,
triangular matrix algebras and
derived discrete algebras, and we prove that
derived discrete algebras satisfy the Fg condition.

Let $A$ be an algebra, and let $e\in A$ be an idempotent
such that $AeA$ is a stratifying ideal, that is,
$Ae\otimes _{eAe}^LeA\cong AeA$ canonically.
From \cite{CPS96}, there is a recollement
$$\xymatrix@!=9pc{ \mathcal{D}(A/AeA)\ar[r]|{- \otimes^L_{A/AeA} A/AeA}
& \mathcal{D}A \ar@<-3ex>[l]|{-\otimes^L_A A/AeA}
\ar@<+3ex>[l] \ar[r]|{-\otimes _A^L Ae} &
\mathcal{D}(eAe) \ar@<-3ex>[l]|{-\otimes^L_{eAe} eA}
\ar@<+3ex>[l]} .$$
By Theorem~\ref{theorem-main-2},
the functor $-\otimes _A^L Ae$ is an
eventually homological isomorphism if and only
if $\pd _{A^e}A/AeA<\infty$. Applying
Theorem~\ref{theorem-Gor} and Theorem~\ref{theorem-Fg},
we recover the following result of Nagase.
\begin{corollary}{\rm (Nagase \cite{Nag11})}\label{cor-idem}
Let $A$ be a finite dimensional
algebra over an algebraically closed field $k$.
Suppose $AeA$ is a stratifying ideal
and $\pd _{A^e}A/AeA<\infty$. Then we have

{\rm (1)} $\HH ^{\geq d}(A)\cong \HH ^{\geq d}(eAe)$,
for some positive integer $d$.

{\rm (2)} $A$ satisfies Fg if and only if so does $eAe$ .

{\rm (3)} $A$ is Gorenstein if and only if so is $eAe$.

\end{corollary}

Let $B$ and $C$ be finite dimensional algebras
over a field $k$, and let $M$ be a finitely
generated $C$-$B$-bimodule.
Then we have the triangular matrix algebra
$A = \left[\begin{array}{cc} B & 0
\\ M & C \end{array}\right] $,
where the addition and the multiplication are given by the ordinary operations on
matrices. From \cite[Example 3.4]{AKLY17}, there is
a recollement $$\xymatrix@!=6pc{ \mathcal{D}B\ar[r]|{- \otimes^L_B e_1A}
& \mathcal{D}A \ar@<-3ex>[l]
\ar@<+3ex>[l] \ar[r]|{-\otimes _A^LAe_2} &
\mathcal{D}C \ar@<-3ex>[l]|{-\otimes^L_C e_2A}
\ar@<+3ex>[l]},$$
where $e_1=\left[\begin{array}{cc} 1 & 0
\\ 0 & 0\end{array}\right] $ and $e_2=\left[\begin{array}{cc} 0 & 0
\\ 0 & 1\end{array}\right] $.
By Theorem~\ref{theorem-main-2}, $-\otimes _A^LAe_2$ is an
eventually homological isomorphism if $
\gl.B<\infty$ and $\pd _C(e_2A)<\infty$, and the latter holds
precisely when $\pd _CM<\infty$. Combining Theorem~\ref{theorem-Gor},
Theorem~\ref{theorem-sing} and Theorem~\ref{theorem-Fg}, we
reobtain the following corollary in \cite{PSS14}.
\begin{corollary}{\rm (\cite[Corollary 8.17]{PSS14})}
Let $A = \left[\begin{array}{cc} B & 0
\\ M & C \end{array}\right] $ be a triangular matrix algebra
over an algebraically closed field $k$.
Suppose $\gl.B<\infty$ and $\pd _CM<\infty$. Then the following
hold.

{\rm (1)} The algebras A and C are singularly equivalent

{\rm (2)} $A$ satisfies Fg if and only if so does $C$ .

{\rm (3)} $A$ is Gorenstein if and only if so is $C$.
\end{corollary}

Also, by \cite[Example 3.4]{AKLY17}, we have another
recollement $$\xymatrix@!=6pc{ \mathcal{D}C\ar[r]
& \mathcal{D}A \ar@<-3ex>[l]
\ar@<+3ex>[l] \ar[r]|{-\otimes^L_A Ae_1} &
\mathcal{D}B \ar@<-3ex>[l]|{-\otimes^L_B e_1A}
\ar@<+3ex>[l]},$$ and similarly,
we recover the following result.
\begin{corollary}{\rm (\cite[Corollary 8.19]{PSS14})}
Let $A = \left[\begin{array}{cc} B & 0
\\ M & C \end{array}\right] $ be a triangular matrix algebra
over an algebraically closed field $k$.
Suppose $\gl.C<\infty$ and $\pd M_B<\infty$. Then the following
hold.

{\rm (1)} {\rm (See \cite[Theorem 4.1]{Chen09})}The algebras A and B are singularly equivalent

{\rm (2)} $A$ satisfies Fg if and only if so does $B$ .

{\rm (3)} $A$ is Gorenstein if and only if so is $B$.
\end{corollary}

In \cite{Nag11}, Gorenstein Nakayama algebras are reduced to
(local) selfinjective Nakayama algebras, via idempotents
in the situation of Corollary~\ref{cor-idem}.
Thus, Gorenstein Nakayama algebras satisfy the
Fg condition because so do selfinjective Nakayama algebras.
Now, we will reduce derived discrete
algebras by recollements with the functor $j^*$ being an
eventually homological isomorphism.

From \cite{Voss01}, an algebra
$A$ is said to be {\it derived discrete} provided for
every positive element $\mathbf{d}  \in
K_0(A)^{(\mathbb{Z})}$ there are only finitely many isomorphism
classes of indecomposable objects $X$ in $\mathcal{D}^b(A)$ of
cohomology dimension vector $(\underline{\dim} H^p(X))_{p \in \mathbb{Z}}
= \mathbf{d}$. Up to derived
equivalent, a basic connected derived
discrete algebra is either a piecewise hereditary
algebra of Dynkin type or a special gentle algebra $\Lambda(r,n,m)$,
see \cite{Voss01, BGS04} for details.

\begin{proposition}\label{prop-der-disc}
Let $A$ be a derived
discrete algebra. Then we have

{\rm (1)} $\mathcal{D}A$ admits a finite stratification
of derived categories along recollements with the functor $j^*$ being an
eventually homological isomorphism, where all derived-simple factors
are either $k$ or $2$-truncated cycle algebras.

{\rm (2)} $A$ satisfies the Fg condition.
\end{proposition}

\begin{proof}
(1): It follows from \cite[Theorem 19]{Q16} that $\mathcal{D}A$ admits a finite stratification
of derived categories along $n$-recollements, where all derived-simple factors
are either $k$ or $2$-truncated cycle algebras. In these recollements
($\mathcal{D} B$,\ $\mathcal{D} A$,\ $\mathcal{D} C$,\ $i^*,i_*,i^!,j_!,j^*,j_*$),
all sixes functors restrict to both $K^b(\proj)$ and $K^b(\inj)$. Moreover,
either $\gl A<\infty $ and hence $\gl B<\infty $, or $\gl A=\infty $ and one
of $B$ and $C$ has finite global dimension (see \cite[proposition 8]{Q16}).
By Theorem~\ref{theorem-main-1}, the functor $j^*$ in
($\mathcal{D} B$,\ $\mathcal{D} A$,\ $\mathcal{D} C$,\ $i^*,i_*,i^!,j_!,j^*,j_*$)
or the functor $i^!$ in
($\mathcal{D} C$,\ $\mathcal{D} A$,\ $\mathcal{D} B$,\ $j^*, j_*, j^\theta,
i_*,i^!, i_\theta$) is
an eventually homological isomorphism.

(2): By \cite[Section 4]{B08}, selfinjective
Nakayama algebras satisfy the Fg condition. Therefore, it follows from (1)
and Theorem~\ref{theorem-Fg} that A satisfies the Fg condition.
\end{proof}

\begin{remark}
{\rm Proposition~\ref{prop-der-disc} add derived
discrete algebras to the classes of algebras where the Fg is known to
hold, e.g. Gorenstein Nakayama algebras \cite{Nag11}, group algebras of finite groups
\cite{E61, V59}
and local finite dimensional algebras
which are complete intersections \cite{G74}.}

\end{remark}

\noindent {\footnotesize {\bf ACKNOWLEDGMENT.} This work is supported by
the National Natural Science Foundation of China (11701321, 11601098) and Yunnan Applied Basic Research
Project 2016FD077.}

\end{document}